\documentclass{amsart}
\usepackage{amsfonts}
\usepackage{graphicx}
\usepackage{amscd}

\newtheorem{theorem}{Theorem}
\theoremstyle{plain}

\newtheorem{conjecture}{Conjecture}
\newtheorem{corollary}{Corollary}

\newtheorem{definition}{Definition}

\newtheorem{lemma}{Lemma}

\newtheorem{proposition}{Proposition}
\newtheorem{remark}{Remark}

\numberwithin{equation}{section}

\input{tcilatex}

\begin{document}
\title[Sampling and interpolation of polyanalytic functions]{Sampling and
interpolation in Bargmann-Fock spaces of polyanalytic functions}
\author{Lu\'{\i}s Daniel Abreu}
\address{CMUC, Department of Mathematics of University of Coimbra, School of
Science and Technology (FCTUC) 3001-454 Coimbra, Portugal \\
and NuHAG, Faculty of Mathematics, University of Vienna, Nordbergstrasse 15,
A-1090 Wien, Austria}
\email{daniel@mat.uc.pt}
\urladdr{http://www.mat.uc.pt/\symbol{126}daniel/}
\thanks{}
\date{November, 2009}
\thanks{This research was partially supported by CMUC/FCT and FCT
post-doctoral grant SFRH/BPD/26078/2005, POCI 2010 and FSE}

\begin{abstract}
Using Gabor analysis, we give a complete characterization of all lattice
sampling and interpolating sequences in the Fock space of polyanalytic
functions, displaying a "Nyquist rate" which increases with $n$, the degree
of polyanaliticity of the space. Such conditions are equivalent to sharp
lattice density conditions for certain vector-valued Gabor systems, namely
superframes and Gabor super-Riesz sequences with Hermite windows, and in the
case of superframes they were studied recently by Gr\"{o}chenig and
Lyubarskii. The proofs of our main results use variations of the
Janssen-Ron-Shen duality principle and reveal a duality between sampling and
interpolation in polyanalytic spaces, and multiple interpolation and
sampling in analytic spaces. To connect these topics we introduce the\emph{\
polyanalytic Bargmann transform}, a unitary mapping between vector valued
Hilbert spaces and polyanalytic Fock spaces, which extends the Bargmann
transform to polyanalytic spaces. Motivated by this connection, we discuss a
vector-valued version of the Gabor transform. These ideas have natural
applications in the context of multiplexing of signals. We also point out
that a recent result of Balan, Casazza and Landau, concerning density of
Gabor frames, has important consequences for the Gr\"{o}chenig-Lyubarskii
conjecture on the density of Gabor frames with Hermite windows.
\end{abstract}

\maketitle

\section{Introduction}

In this paper we find a new, and perhaps unexpected, connection between
polyanalytic functions and time-frequency analysis, and use it to obtain a
complete characterization of all lattice sampling and interpolating
sequences in the Bargmann-Fock space of polyanalytic functions or,
equivalently, of all lattice vector-valued Gabor frames and vector-valued
Gabor Riesz sequences for $L^{2}(%
\mathbb{R}
,%
\mathbb{C}
^{n})$.

\subsection{Overview}

The Bargmann-Fock space of polyanalytic functions, $\mathbf{F}^{n}(%
\mathbb{C}
^{d})$, consists of all functions satisfying the equation
\begin{equation}
\left( \frac{d}{d\overline{z}}\right) ^{n}F(z)=0\text{,}
\label{defpolyanalytic}
\end{equation}%
and such that%
\begin{equation}
\int_{%
\mathbb{C}
^{d}}\left\vert F(z)\right\vert ^{2}e^{-\pi \left\vert z\right\vert
^{2}}dz<\infty \text{.}  \label{growth}
\end{equation}%
\emph{\ }Functions satisfying (\ref{defpolyanalytic}) are known as \emph{%
polyanalytic functions} of order $n$. Since (\ref{defpolyanalytic})
generalizes the Cauchy-Riemann equation
\begin{equation*}
\frac{d}{d\overline{z}}F(z)=0\text{,}
\end{equation*}%
then the space $\mathbf{F}^{n}(%
\mathbb{C}
^{d})$ is a generalization of the Bargmann-Fock space of analytic functions,
$\mathcal{F}(%
\mathbb{C}
^{d})=\mathbf{F}^{1}(%
\mathbb{C}
^{d})$. In the case of $\mathcal{F}(%
\mathbb{C}
)$, a complete description of the sets of sampling and interpolation is
known \cite{Ly},\cite{S1},\cite{SW}.

Polyanalytic functions inherit some of the properties of analytic functions,
often in a nontrivial form. However, as in the theory of several complex
variables, many of the properties break down once we leave the analytic
setting. An obvious difference lies in the structure of the zeros. For
instance, while nonzero entire functions do not have sets of zeros with an
accumulation point, polyanalytic functions can vanish along closed curves:
just take $F(z)=\overline{z}z-1=\left\vert z\right\vert ^{2}-1$, a
polyanalytic function of order 2. Polyanalytic functions have been
investigated thoroughly, notably by Balk and his students \cite{Balk}. They
are naturally related to polyharmonic functions, which have an intriguing
structure of zero sets \cite{HayKor}, \cite{BalkMazalov}, \cite{Render}.

We will study the spaces $\mathbf{F}^{n}(%
\mathbb{C}
^{d})$ using time-frequency analysis, and think about the polyanalytic
functions in a different way from the classical classical approach.

The link to time-frequency analysis has impressive consequences: by endowing
the spaces $\mathbf{F}^{n}(%
\mathbb{C}
^{d})$ with the structure inherent to translations and modulations, one can
use tools that were unavailable with complex variables. This will provide $%
\mathbf{F}^{n}(%
\mathbb{C}
^{d})$ with properties reminiscent of the classical analytic Fock space.

In order to state our main results, we briefly recall some definitions. See
sections 4 and 5 for more details. We associate the sequence $\Lambda
=\{(x,w)\}$ with the sequence of complex numbers $\Gamma =\{(x+iw)\}$. Then $%
\Gamma $ is \emph{an interpolating sequence} for $\mathbf{F}^{n}(%
\mathbb{C}
^{d})$ if, for every sequence $\{\alpha _{i,j}\}\in l^{2}$, there exists $%
F\in \mathbf{F}^{n}(%
\mathbb{C}
^{d})$ such that
\begin{equation*}
e^{i\pi x\omega -\frac{\pi }{2}\left\vert z\right\vert ^{2}}F(z)=\alpha
_{i,j},
\end{equation*}%
for every $z\in \Gamma .$ We say that $\Gamma $ is \emph{a} \emph{sampling
sequence} for $\mathbf{F}^{n}(%
\mathbb{C}
^{d})$\ if there exist $A,B>0$ such that, for every $F\in \mathbf{F}^{n}(%
\mathbb{C}
^{d})$,%
\begin{equation*}
A\left\Vert F\right\Vert _{\mathbf{F}^{n}(%
\mathbb{C}
^{d})}^{2}\leq \sum_{z\in \Gamma }\left\vert F(z)\right\vert ^{2}e^{-\pi
\left\vert z\right\vert ^{2}}\leq B\left\Vert F\right\Vert _{\mathbf{F}^{n}(%
\mathbb{C}
^{d})}^{2}.
\end{equation*}

The concept of interpolating sequences has its roots in deep problems in
complex analysis, and sampling sequences are a major issue in signal
processing, since they correspond to the case where stable numerical
reconstructions of a function \ from its samples is possible. Monograph \cite%
{Seipmonograph} is a good introduction to sampling and interpolation and its
interconnections with other branches of pure and applied mathematics.

Our main results, Theorem 4 and Theorem 6 below, use the concept of Beurling
density, which, in the lattice case, is given by $D(\Gamma )=D(\Lambda
)=\left\vert \det A\right\vert ^{-1}$, where $\Lambda =A%
\mathbb{Z}
^{2}$.%
\begin{equation*}
\end{equation*}%
\textbf{Theorem 4. }The lattice $\Gamma $ is a sampling sequence for $%
\mathbf{F}^{n}(%
\mathbb{C}
)$ if and only if:
\begin{equation*}
D(\Gamma )>n.
\end{equation*}%
\textbf{Theorem 6. }The lattice $\Gamma $ is an interpolating sequence for $%
\mathbf{F}^{n}(%
\mathbb{C}
)$ if and only if:
\begin{equation*}
D(\Gamma )<n.
\end{equation*}

These results follow by establishing one duality between sampling in $%
\mathbf{F}^{n}(%
\mathbb{C}
^{d})$ and multiple interpolation in $\mathcal{F}(%
\mathbb{C}
^{d})$ and a second duality between interpolation in $\mathbf{F}^{n}(%
\mathbb{C}
^{d})$ and multiple sampling in $\mathcal{F}(%
\mathbb{C}
^{d})$. When $d=1$, these properties allow us to directly apply the results
in \cite{Brekkeseip}. Taking $n=1$ we recover the well known duality between
sampling and interpolation in $\mathcal{F}(%
\mathbb{C}
^{d})$.

Theorems exhibiting a "Nyquist rate" phenomenon tend to be hard to prove.
They have been studied, for general sequences, in spaces of analytic
functions, first in the Paley-Wiener space \cite{Beurling},\cite{Landauacta},%
\cite{LandauIEE} and then in Bargmann-Fock \cite{Ly},\cite{S1},\cite{SW} and
Bergman \cite{Seip2} spaces of analytic functions. There are two reasons for
us to believe that Beurling type methods do not work here. First, they use
complex variables tools that are not available in the polyanalytic
situation. Second, in all of these spaces, there were known "doubly
orthogonal systems" \cite{Seipdoubly}, which provided the eigenfunctions for
the fundamental equation involving the "concentration operator". We are not
aware of a system with this double orthogonality property in the
polyanalytic situation.

Therefore, we introduce new tools.

With a view to relating our problem to one concerning the density of vector
valued Gabor systems, we extend Bargmann%
\'{}%
s work \cite{Bar} to the setting of polyanalytic functions. Once we do this,
our argument, which is conceptual in nature, follows smoothly.

It is also worth noting that the density Theorem in Gabor analysis has
itself a very rich story, beginning with fundamental but imprecise
statements by John Von Neumann and Dennis Gabor, which caught the attention
of mathematicians after conjectures by Daubechies and Grossman \cite%
{DauGross}. See the survey article \cite{Heil}, \cite{CHO} and the important
special windows studied in \cite{JanStr}.

To give a context to our approach, recall the connection between the
classical (analytic) Bargmann-Fock space and time-frequency analysis.

It is well known that, up to a certain weight, the Gabor transform with a
gaussian window belongs to the Fock space of analytic functions. Moreover,
it has been shown that this is the only choice leading to spaces of analytic
functions \cite{AscensiBruna}.

However, a nice picture emerges when we take Hermite functions as windows.
The analytic situation generated by the gaussian window then becomes the tip
of the iceberg of a larger structure involving spaces of polyanalytic
functions. Indeed, the Gabor transform with the $nth$ Hermite function is,
up to a certain weight (the same as in the analytic case), a polyanalytic
function of order $n+1$.

To fully understand the situation, we will need the spaces constituted by
the functions satisfying (\ref{growth}), which are polyanalytic of order $n$%
, but are \emph{not} polyanalytic of any lower order (in particular they
have no analytic functions). These are the \emph{true} polyanalytic Fock
spaces $\mathcal{F}^{n}(%
\mathbb{C}
^{d})$. The polyanalytic Fock and true polyanalytic Fock spaces are related
by the following orthogonal decomposition (see Corollary 1 in section 3):%
\begin{equation*}
\mathbf{F}^{n}(%
\mathbb{C}
^{d})=\mathcal{F}^{0}(%
\mathbb{C}
^{d})\oplus ...\oplus \mathcal{F}^{n-1}(%
\mathbb{C}
^{d})\text{.}
\end{equation*}

Then, each space $\mathcal{F}^{n}(%
\mathbb{C}
^{d})$ is associated with Gabor transforms with the $nth$ Hermite window.
Such occurrence, which seems to have been hitherto unnoticed, will be
fundamental in our discussion. This observation is related to some recent
developments in Gabor analysis with Hermite functions \cite{CharlyYura},\cite%
{CharlyYurasuper},\cite{Fuhr}, to Janssen%
\'{}%
s approach to the density Theorem \cite{Janssen},\cite{J3} and also to the
techniques used in \cite{Hutnikcr},\cite{Hutnikm},\cite{VasiBergman}, which
suggest that wavelet spaces and polyanalytic functions share intriguing
patterns.

Fock spaces of polyanalytic functions are briefly mentioned in Balk%
\'{}%
s monograph \cite{Balk} and they are implicit in quantum mechanics, in
connection with the Landau levels of the Schr\"{o}dinger operator with
magnetic field \cite{Shigekawa},\cite{LuefGosson} and displaced Fock states
\cite{Wunsche}. However, we were not able to find any reference to
polyanalytic functions in the mathematical physics literature, apart from
\cite{VasiFock}, where creation and annihilation operators are used.

\subsection{The results of Gr\"{o}chenig-Lyubarskii and of
Balan-Casazza-Landau}

Our results are connected to a very recent result of Gr\"{o}chenig and
Lyubarskii, which deserves more specific comment in this introduction.
Denote by $G(\mathbf{h}_{n},\Lambda )$ the set of the translations and
modulations indexed by the lattice $\Lambda $ and acting coordinate-wise on
the vector $\mathbf{h}_{n}.$ Then, the following result holds.

\textbf{Theorem }\cite{CharlyYurasuper}: \emph{Let }$\mathbf{h}%
_{n}=(h_{0},...,h_{n-1})$\emph{\ be the vector of the first }$n$\emph{\
Hermite functions. Then }$G(\mathbf{h}_{n},\Lambda )$\emph{\ is a frame for }%
$L^{2}(%
\mathbb{R}
,%
\mathbb{C}
^{n})$\emph{\ if and only if }%
\begin{equation*}
D(\Lambda )>n.
\end{equation*}%
One may wonder if the equivalence of this condition to the one in Theorem 4
reflects causality or casuality. As we will see, \emph{causality} is the
answer. Actually, one of the key steps in our approach consists of showing
that Theorem 4 is equivalent to the above Theorem.

The original proof of this Theorem in \cite{CharlyYurasuper} combines the
use of the so-called Wexler-Raz biorthogonality relations with complex
analysis techniques based on properties of the Weierstrass sigma function.
We will give an alternative proof, which is considerably shorter (at the
cost of using deeper results from the literature) and has the advantage of
also characterizing the vector valued Gabor Riesz sequences with Hermite
windows in the following statement, which is equivalent to Theorem 6:

\textbf{Theorem 7: }$\mathcal{G}(\mathbf{h}_{n},\Lambda )$ \emph{is a Riesz
sequence for }$L^{2}(%
\mathbb{R}
,%
\mathbb{C}
^{n})$\emph{\ if and only if }%
\begin{equation*}
D(\Gamma )<n\emph{.}
\end{equation*}

One should also remark that the necessity of the condition $D(\Lambda )\geq
n $ for vector valued frames, follows from a result of Balan \cite%
{BalanDensity}. Moreover, this condition holds for general sets, since a
Ramanathan-Steger \cite{RS} \ type argument is used.

Now let us look closer at the problem of deciding when the translations and
modulations of a \emph{single Hermite }function constitute a frame.

It can be easily seen as a corollary of the above Theorem that $D(\Lambda
)>n $ is sufficient for the system $G(h_{n-1},\Lambda )$ to be a frame. It
has been observed by Gr\"{o}chenig and Lyubarskii that there are some
examples which support the intriguing conjecture that such a result might be
sharp.

A recent result of Balan, Casazza and Landau shows that this conjecture
cannot hold for general sets. Let $S_{0}$ stand for the \emph{Feichtinger
algebra }\cite{Fei}.

\textbf{Theorem \cite{BCL} }\emph{Assume }$G(g,\Lambda )$\emph{\ is a Gabor
frame for }$L^{2}\left(
\mathbb{R}
^{d}\right) $\emph{\ with }$g\in S_{0}$\emph{. Then, for every }$\epsilon >0$%
\emph{, there exists a subset }$J_{\epsilon }\subset \Lambda $\emph{\ so
that }$G(g,J_{\epsilon })$\emph{\ is a Gabor frame for }$L^{2}\left(
\mathbb{R}
^{d}\right) $\emph{\ and its upper Beurling density satisfies }$%
D^{+}(J_{\epsilon })\leq 1+\epsilon $\emph{.}

This Theorem implies that Gr\"{o}chenig-Lyubarskii conjecture can only be
true for lattices. Indeed, for every $n$, the $S_{0}$ norm of the Hermite
functions (see \cite{JanssenHermite}) is%
\begin{equation*}
\left\Vert h_{n}\right\Vert _{S_{0}}=2^{\frac{n}{2}+1}\frac{\Gamma (\frac{1}{%
2}n+1)}{\sqrt{n!}},
\end{equation*}%
and therefore,
\begin{equation*}
h_{n}\in S_{0}.
\end{equation*}%
Thus, the Balan-Casazza-Landau Theorem tells us that we can always find, for
every $n$, a set $J$ with Beurling density arbitrarily close to one and such
that $G(h_{n},\Lambda )$ is a frame. However, it is still possible that the
conjecture is true for lattices, since a subset of a lattice may not be a
lattice.

\begin{conjecture}
Let $\Lambda $ be a lattice in $%
\mathbb{R}
^{2}$. If $G(h_{n},\Lambda )$ is s Gabor frame for $L^{2}\left(
\mathbb{R}
^{d}\right) $, then%
\begin{equation*}
D(\Lambda )>n.
\end{equation*}
\end{conjecture}

\subsection{New concepts}

We will introduce some new concepts troughout this paper, in order to extend
Bargmann%
\'{}%
s theory \cite{Bar} and its connection to Gabor analysis in the polyanalytic
setting.

\begin{itemize}
\item We first introduce what we call the \emph{true-polyanalytic Bargmann
transform}:
\begin{equation*}
(\mathcal{B}^{n}f)(z)=(\pi ^{\left\vert n\right\vert }n!)^{-\frac{1}{2}%
}e^{\pi \left\vert z\right\vert ^{2}}\frac{d^{n}}{dz^{n}}\left[ e^{-\pi
\left\vert z\right\vert ^{2}}F(z)\right] \text{.}
\end{equation*}%
Here $F$ stands for the Bargmann transform of $f$. As we will see, this is a
unitary mapping from $L^{2}(%
\mathbb{R}
^{d})$ to $\mathcal{F}^{n}(%
\mathbb{C}
^{d})$. This mapping relates to Gabor transforms with Hermite windows $\Phi
_{n}$ in the following way:%
\begin{equation*}
V_{\Phi _{n}}f(x,\omega )=e^{i\pi x\omega -\pi \frac{\left\vert z\right\vert
^{2}}{2}}(\mathcal{B}^{n}f)(z)\text{,}
\end{equation*}%
and we will provide its basic theory starting from\ this relation, following
as much as possible the presentation of the Bargmann transform given in Gr%
\"{o}chenig%
\'{}%
s book \cite[section 3.3]{Charly}.

\item For vector-valued functions $\mathbf{f}=(f_{0},...,f_{n-1})$, we
define the \emph{polyanalytic Bargmann} \emph{transform},
\begin{equation*}
(\mathbf{B}^{n}\mathbf{f})=\sum_{0\leq k\leq n-1}(\mathcal{B}^{k}f_{k}),
\end{equation*}%
which will be unitary between $L^{2}(%
\mathbb{R}
^{d},%
\mathbb{C}
^{n})$ and $\mathbf{F}^{n}(%
\mathbb{C}
^{d})$.

\item In the last section we will see that the polyanalytic Bargmann
transform is a special case of a vector-valued version of the Gabor
transform. Although this transform plays no direct role in the proofs of the
main Theorems, it must be in the picture for completeness, since it is the
natural time-frequency transformation of which the polyanalytic Bargmann
transform is a special case:
\begin{equation*}
\mathbf{V}_{\mathbf{g}}\mathbf{f}(x,\omega
)=\sum_{k=0}^{n-1}V_{g_{k}}f_{k}(x,\omega ).
\end{equation*}%
In the case where $\{g_{k}\}_{k=0,...n-1}$ constitutes an orthonormal
sequence, it provides an isometry between $L^{2}(%
\mathbb{R}
^{d},%
\mathbb{C}
^{n})$ and between $L^{2}(%
\mathbb{R}
^{2d})$. This transform is the continuous counterpart of the \emph{%
superframes} discussed in \cite{DL}, \cite{Balan}, \cite[Theorem 2.7]%
{CharlyYurasuper}.
\end{itemize}

\subsection{\textbf{Technical summary} \textbf{of proofs }}

With the tools described above at hand, our main argument will depend on two
profound results. More specifically, we will combine variations on the
Janssen-Ron-Shen duality principle \cite{RonShen} with the characterization
of multiple sampling and interpolation sequences in the Fock space \cite%
{Brekkeseip}. The duality principles reflect all the rich inner structure of
Gabor frames. The second result uses a deep elaboration on Beurling%
\'{}%
s balayage technique \cite{Beurling} developed by Seip in \cite{Seip2}.

We will proceed as follows.

First, using an orthogonal basis for the polyanalytic Fock spaces, we prove
the unitarity of $\mathcal{B}^{n}$ and $\mathbf{B}^{n}$. Then we study
sampling in $\mathbf{F}^{n}(%
\mathbb{C}
)$. Using the unitary mapping $\mathbf{B}^{n}$ we show that the problem is
equivalent to the study of vector valued frames with Hermite windows, also
known as superframes \cite{Balan},\cite{CharlyYurasuper}. This problem has
been recently studied in \cite{CharlyYurasuper}, but we provide an
alternative proof, which is more natural in the context of sampling and
interpolation: applying a vector valued version of the Janssen-Ron-Shen
duality we translate the statement into a problem concerning unions of Riesz
sequences. After noticing that the latter is equivalent to a multiple
interpolation problem in Fock spaces of analytic functions, we apply the
interpolation result in \cite{Brekkeseip}. We then study interpolation in $%
\mathbf{F}^{n}(%
\mathbb{C}
)$. In order to do this, we "dualize" the arguments that we have used in the
sampling part, once again using the Janssen-Ron-Shen duality, this time
between vector-valued Riesz sequences and multi-frames with Hermite
functions. This translates our interpolation problem into one of multiple
sampling. Noticing that this problem is equivalent to multiple sampling in
Fock spaces, we apply the sampling result from \cite{Brekkeseip}.

\subsection{\textbf{Organization of the paper}}

The next section contains the classical tools that we are going to use. We
list the basic properties of the Gabor transform, the Bargmann transform and
the Hermite functions.

In the third section, we introduce the true polyanalytic Bargmann and the
polyanalytic Bargmann transforms. By making a connection with the Gabor
transform, we study their basic properties, find an orthogonal basis for the
polyanalytic Fock spaces and prove the unitarity properties.

\ Our main results are in the fourth and fifth sections, where we derive the
duality principles and study sampling and interpolation for $\mathbf{F}^{n}(%
\mathbb{C}
)$.

In our last section we introduce the super Gabor transform and make some
remarks of a more informal character concerning applications and open
problems.

\section{Background}

\subsection{The Gabor transform}

Fix a function $g\neq 0$. Then the Gabor (short-time) Fourier transform of a
function $f$ with respect to the \textquotedblright
window\textquotedblright\ $g$ is defined, for every $x,\omega \in
\mathbb{R}
^{d}$, as
\begin{equation}
V_{g}f(x,\omega )=\int_{%
\mathbb{R}
^{d}}f(t)\overline{g(t-x)}e^{-2\pi it\omega }dt.  \label{Gabor}
\end{equation}%
The following relations are usually called \emph{the orthogonal relations
for the short-time Fourier transform}. Let $f_{1},f_{2},g_{1},g_{2}\in L^{2}(%
\mathbb{R}
^{d})$. Then $V_{g_{1}}f_{1},V_{g_{2}}f_{2}\in L^{2}(%
\mathbb{R}
^{2d})$ and
\begin{equation}
\left\langle V_{g_{1}}f_{1},V_{g_{2}}f_{2}\right\rangle _{L^{2}(%
\mathbb{R}
^{2d})}=\left\langle f_{1},f_{2}\right\rangle _{L^{2}(%
\mathbb{R}
^{d})}\overline{\left\langle g_{1},g_{2}\right\rangle _{L^{2}(%
\mathbb{R}
^{d})}}\text{.}  \label{ortogonalityrelations}
\end{equation}

The Gabor transform provides an isometry
\begin{equation*}
V_{g}:L^{2}(%
\mathbb{R}
^{d})\rightarrow L^{2}(%
\mathbb{R}
^{2d})\text{,}
\end{equation*}%
that is, if $f,g\in L^{2}(%
\mathbb{R}
^{d})$, then
\begin{equation}
\left\Vert V_{g}f\right\Vert _{L^{2}(%
\mathbb{R}
^{2d})}=\left\Vert f\right\Vert _{L^{2}(%
\mathbb{R}
^{d})}\left\Vert g\right\Vert _{L^{2}(%
\mathbb{R}
^{d})}\text{.}  \label{Gabor isometry}
\end{equation}%
For every $x,\omega \in
\mathbb{R}
^{d}$ define the operators translation by $x$ and modulation by $\omega $ as
\begin{eqnarray*}
T_{x}f(t) &=&f(t-x), \\
M_{\omega }f(t) &=&e^{2\pi i\omega t}f(t).
\end{eqnarray*}%
Using these operators we can write (\ref{Gabor}) as%
\begin{equation*}
V_{g}f(x,\omega )=\left\langle f,M_{\omega }T_{x}g\right\rangle _{L^{2}(%
\mathbb{R}
^{d})}\text{.}
\end{equation*}

\subsection{The Bargmann transform}

Here we will use multi-index notation: $z=(z_{1},...z_{d})$, $%
n=(n_{1},...,n_{d})$ and $\left\vert n\right\vert =n_{1}+...+n_{d}$. The
Bargmann transform, defined by
\begin{equation*}
(\mathcal{B}f)(z)=\int_{%
\mathbb{R}
^{d}}f(t)e^{2\pi tz-\pi z^{2}-\frac{\pi }{2}t^{2}}dt\text{,}
\end{equation*}%
is an isomorphism
\begin{equation*}
\mathcal{B}:L^{2}(%
\mathbb{R}
^{d})\rightarrow \mathcal{F}(%
\mathbb{C}
^{d}),
\end{equation*}%
where $\mathcal{F}(%
\mathbb{C}
^{d})$ stands for the Bargmann-Fock space of analytic functions in $%
\mathbb{C}
^{d}$ with the norm
\begin{equation}
\left\Vert F\right\Vert _{\mathcal{F}(%
\mathbb{C}
^{d})}^{2}=\int_{%
\mathbb{C}
^{d}}\left\vert F(z)\right\vert ^{2}e^{-\pi \left\vert z\right\vert ^{2}}dz.
\label{norm}
\end{equation}%
The collection of the monomials of the form
\begin{equation}
e_{n}(z)=\left( \frac{\pi ^{\left\vert n\right\vert }}{n!}\right) ^{\frac{1}{%
2}}z^{n}=\prod_{j=1}^{d}\frac{\pi ^{\frac{n_{j}}{2}}}{\sqrt{n_{j}!}}z^{n_{j}}%
\text{,}  \label{ort}
\end{equation}%
where $n=(n_{1,...}n_{d})$, with $n_{i}\geq 0$, constitutes an orthonormal
basis of $\mathcal{F}(%
\mathbb{C}
^{d})$. The reproducing kernel of $\mathcal{F}(%
\mathbb{C}
^{d})$ is the function $e^{\pi \overline{z}w}$. This means that, for every $%
F\in \mathcal{F}(%
\mathbb{C}
^{d})$,%
\begin{equation*}
\left\langle F(w),e^{\pi \overline{z}w}\right\rangle _{\mathcal{F}(%
\mathbb{C}
^{d})}=F(z)\text{.}
\end{equation*}%
Differentiating $n-k$ times the corresponding reproducing equation, we obtain%
\begin{equation}
\left\langle F(w),w^{n-k}e^{\pi \overline{z}w}\right\rangle _{\mathcal{F}(%
\mathbb{C}
^{d})}=\pi ^{k-n}F^{(n-k)}(z),  \label{repder}
\end{equation}%
where for $d>1$ we are using the multi-index derivative%
\begin{equation*}
\frac{d}{dz}f=\frac{df}{dz_{1}...dz_{d}}.
\end{equation*}

A simple calculation shows that the Bargmann transform is related to the
Gabor transform with the Gaussian window $\varphi (t)=2^{\frac{d}{4}}e^{-\pi
t^{2}}$ by the formula
\begin{equation}
V_{\varphi }f(x,-\omega )=e^{i\pi x\omega -\pi \frac{\left\vert z\right\vert
^{2}}{2}}(\mathcal{B}f)(z)\text{,}  \label{formulaBar}
\end{equation}%
where $z=x+i\omega $.

We will need one more operator. Define a "translation" $\beta _{\zeta }$ on $%
\mathcal{F}(%
\mathbb{C}
^{d})$ by
\begin{equation}
\beta _{z}F(\zeta )=e^{i\pi x\omega -\pi \frac{\left\vert z\right\vert ^{2}}{%
2}}e^{\pi \overline{z}\zeta }F(\zeta -z)\text{. }  \label{shift}
\end{equation}%
The operator $\beta _{z}$ satisfies the intertwining property
\begin{equation}
\beta _{z}\mathcal{B}=\mathcal{B}M_{\omega }T_{x}\text{, \ \ \ \ }%
z=x+i\omega .  \label{intertwining}
\end{equation}

\subsection{The Hermite functions}

The \emph{Hermite functions }can be defined via the so called Rodrigues
Formula
\begin{equation*}
h_{n}(t)=c_{n}e^{\pi t^{2}}\left( \frac{d}{dt}\right) ^{n}\left( e^{-2\pi
t^{2}}\right) .
\end{equation*}%
where $c_{n}$ is chosen in such a way that they can provide an orthonormal
basis of $L^{2}(%
\mathbb{R}
)$. Now let $n=(n_{1},...,n_{d})$ and $x\in
\mathbb{R}
^{d}$. The $d$\emph{-dimensional Hermite functions} are%
\begin{equation*}
\Phi _{n}(x)=\prod_{j=1}^{d}h_{n_{j}}(x_{j})\text{.}
\end{equation*}%
They form a complete orthonormal system of $L^{2}(%
\mathbb{R}
^{d})$.

A very important property of the Hermite functions (see for instance \cite%
{JanssenHermite}) is that they are mapped onto a basis of the Bargmann-Fock
space via the Bargmann transform.
\begin{equation}
(\mathcal{B}\Phi _{n})(z)=e_{n}(z).  \label{BargHermite}
\end{equation}

\section{Polyanalytic Fock spaces and polyanalytic Bargmann transforms}

\subsection{Definitions}

In this section we use multi-index notation in such a way that there will be
no difference between the one and the $d$-dimensional case. Only at the end
of the last two sections is it necessary to specialize $d=1$.

It is well known \cite{Balk} that every polyanalytic function of order $n$
can be uniquely expressed in the form
\begin{equation}
F(z)=\sum_{0\leq k\leq n-1}\overline{z}^{k}\varphi _{k}(z),
\label{polyexpression}
\end{equation}%
where $\{\varphi _{p}(z)\}_{p=0}^{n-1}$ are analytic functions, each of them
with a power series expansion,%
\begin{equation*}
\varphi _{p}(z)=\sum_{j\geq 0}c_{j,p}z^{j}\text{,}
\end{equation*}%
As a result, there is also a power series expansion for the polyanalytic
function $F$:
\begin{equation}
F(z)=\sum_{0\leq p\leq n-1}\overline{z}^{p}\sum_{j\geq 0}c_{j,p}z^{j}\text{.}
\label{seriespoly}
\end{equation}%
We will often use the inner product in the polyanalytic Fock space, given by%
\begin{equation*}
\left\langle F,G\right\rangle _{\mathbf{F}^{n}(%
\mathbb{C}
^{d})}=\int_{%
\mathbb{C}
^{d}}F(z)\overline{G(z)}e^{-\pi \left\vert z\right\vert ^{2}}dz.
\end{equation*}%
Observe also that this implies%
\begin{equation*}
\left\langle F,G\right\rangle _{\mathbf{F}^{n}(%
\mathbb{C}
^{d})}=\left\langle e^{-\pi \frac{\left\vert z\right\vert ^{2}}{2}}F,e^{-\pi
\frac{\left\vert z\right\vert ^{2}}{2}}G\right\rangle _{L^{2}(%
\mathbb{R}
^{2d})}\text{.}
\end{equation*}

\subsection{The true polyanalytic Bargmann transform}

\begin{definition}
The \emph{true polyanalytic Bargmann transform of order n}, of a function on
$%
\mathbb{R}
^{d}$, is defined by the formula
\begin{equation}
(\mathcal{B}^{n}f)(z)=(\pi ^{\left\vert n\right\vert }n!)^{-\frac{1}{2}%
}e^{\pi \left\vert z\right\vert ^{2}}\frac{d^{n}}{dz^{n}}\left[ e^{-\pi
\left\vert z\right\vert ^{2}}F(z)\right] ,  \label{polyBargman}
\end{equation}%
where $F(z)=(\mathcal{B}f)(z)$.
\end{definition}

Clearly $\mathcal{B}^{0}f=\mathcal{B}f$ and $\mathcal{B}^{n}$ is a
generalization of the Bargmann transform. We now provide the fundamental
properties of $\mathcal{B}^{n}$. We try to stay as close as possible to the
presentation of section 3.4 in \cite{Charly}. The next proposition is the
departing point of our study.

\begin{proposition}
If $f$ is a function on $%
\mathbb{R}
^{d}$ with polynomial growth, then its true polyanalytic Bargmann transform $%
\mathcal{B}^{n}f$ is a polyanalytic function of order $n+1$ on $%
\mathbb{C}
^{d}.$ If we write $z=x+i\omega $, then this transform is related to the
Gabor transform with Hermite windows in the following way:
\begin{equation}
V_{\Phi _{n}}f(x,\omega )=e^{i\pi x\omega -\pi \frac{\left\vert z\right\vert
^{2}}{2}}(\mathcal{B}^{n}f)(z)\text{.}  \label{rel}
\end{equation}%
Moreover, if $f\in L^{2}(%
\mathbb{R}
)$, then
\begin{equation}
\left\Vert \mathcal{B}^{n}f\right\Vert _{L^{2}(%
\mathbb{C}
^{d},e^{-\pi \left\vert z\right\vert ^{2}})}=\left\Vert f\right\Vert _{L^{2}(%
\mathbb{R}
^{d})}.  \label{polybfisometry}
\end{equation}
\end{proposition}

\begin{proof}
Let $F=\mathcal{B}f$. The following calculation is from Proposition 3.2 in
\cite{CharlyYura}, where (\ref{repder}) is used:
\begin{eqnarray*}
V_{\Phi _{n}}f(x,\omega ) &=&\left\langle f,M_{x}T_{\omega }\Phi
_{n}\right\rangle _{L^{2}(%
\mathbb{R}
^{d})} \\
&=&\left\langle F,\beta _{z}\mathcal{B}\Phi _{n}\right\rangle _{\mathcal{F}(%
\mathbb{C}
^{d})} \\
&=&(\pi ^{\left\vert n\right\vert }n!)^{-\frac{1}{2}}e^{i\pi x\omega -\frac{%
\pi }{2}\left\vert z\right\vert ^{2}}\left\langle F(w),e^{\pi \overline{z}%
w}(w-z)^{n}\right\rangle _{\mathcal{F}(%
\mathbb{C}
^{d})} \\
&=&(\pi ^{\left\vert n\right\vert }n!)^{-\frac{1}{2}}e^{i\pi x\omega -\frac{%
\pi }{2}\left\vert z\right\vert ^{2}}\sum_{0\leq k\leq n}\binom{n}{k}(-\pi
\overline{z})^{k}F^{(n-k)}(z)\text{.}
\end{eqnarray*}%
Now, since the Bargmann transform of a function in $%
\mathbb{R}
^{d}$ is an entire function \cite[Proposition 3.4.1]{Charly}, the functions $%
F^{(n-k)}(z)$ are also entire, and from (\ref{polyexpression}) we recognize
the sum as a polyanalytic function of order $n+1$. To prove (\ref{rel})
observe that the last expression can be written as
\begin{equation*}
e^{i\pi x\omega -\pi \frac{\left\vert z\right\vert ^{2}}{2}}(\pi
^{\left\vert n\right\vert }n!)^{-\frac{1}{2}}e^{\pi \left\vert z\right\vert
^{2}}\frac{d^{n}}{dz^{n}}\left[ e^{-\pi \left\vert z\right\vert ^{2}}F(z)%
\right] =e^{i\pi x\omega -\pi \frac{\left\vert z\right\vert ^{2}}{2}}(%
\mathcal{B}^{n}f)(z)\text{.}
\end{equation*}%
The isometric property (\ref{polybfisometry}) is an immediate consequence of
(\ref{rel}) and (\ref{Gabor isometry}).
\end{proof}

\subsection{Orthogonal decomposition}

\begin{definition}
For $k,m\in
\mathbb{N}
_{0}^{d}$, consider the functions $e_{k,m}$ defined as
\begin{equation}
e_{k,m}(z)=(\pi ^{\left\vert k\right\vert }k!)^{-\frac{1}{2}}e^{\pi
\left\vert z\right\vert ^{2}}\left( \frac{d}{dz}\right) ^{k}\left[ e^{-\pi
\left\vert z\right\vert ^{2}}e_{m}(z)\right] \text{.}  \label{basis}
\end{equation}
\end{definition}

From (\ref{BargHermite}) one can easily see that
\begin{equation}
e_{k,m}(z)=(\mathcal{B}^{k}\Phi _{m})(z)\text{.}  \label{Barbasis}
\end{equation}

\begin{proposition}
The set $\left\{ e_{k,m}\right\} _{0\leq k\leq n-1;\text{...}m\geq 0}$is an
orthogonal basis of $\mathbf{F}^{n}(%
\mathbb{C}
^{d})$.
\end{proposition}

\begin{proof}
The orthogonality follows from (\ref{Barbasis}), (\ref{rel}) and (\ref%
{ortogonalityrelations}), since
\begin{eqnarray*}
\left\langle e_{k,m},e_{l,j}\right\rangle _{L^{2}(%
\mathbb{R}
^{2d})} &=&\left\langle \mathcal{B}^{k}\Phi _{m},\mathcal{B}^{l}\Phi
_{j}\right\rangle _{\mathcal{F}(%
\mathbb{C}
^{d})} \\
&=&\left\langle e^{\pi \frac{\left\vert z\right\vert ^{2}}{2}-i\pi x\omega
}V_{\Phi _{k}}\Phi _{m},e^{\pi \frac{\left\vert z\right\vert ^{2}}{2}-i\pi
x\omega }V_{\Phi _{l}}\Phi _{j}\right\rangle _{\mathcal{F}(%
\mathbb{C}
^{d})} \\
&=&\left\langle V_{\Phi _{k}}\Phi _{m},V_{\Phi _{l}}\Phi _{j}\right\rangle
_{L^{2}(%
\mathbb{R}
^{2d})} \\
&=&\left\langle \Phi _{m},\Phi _{j}\right\rangle _{L^{2}(%
\mathbb{R}
^{d})}\overline{\left\langle \Phi _{k},\Phi _{l}\right\rangle _{L^{2}(%
\mathbb{R}
^{d})}} \\
&=&\delta _{m,j}\delta _{k,l}.
\end{eqnarray*}%
To prove completeness of $\{e_{k,m}\}$ in $\mathbf{F}^{n}(%
\mathbb{C}
^{d})$, suppose that $F\in \mathbf{F}^{n}(%
\mathbb{C}
^{d})$ is such that
\begin{equation*}
\left\langle F,e_{k,m}\right\rangle _{\mathcal{F}(%
\mathbb{C}
^{d})}=0\text{, \ \ \ \ \ }0\leq k\leq n-1;\text{ \ \ \ \ }m\geq 0\text{.}
\end{equation*}%
For $k=0$, we can use the representation of $F$ in power series (\ref%
{seriespoly}). Interchanging the sums with the integrals and using the
orthogonality of the functions (\ref{ort}), the result is
\begin{equation}
\left\langle F,e_{0,m}\right\rangle _{\mathcal{F}(%
\mathbb{C}
^{d})}=\sum_{0\leq p\leq n-1}c_{p+m,p}\frac{(p+m)!}{\sqrt{m!}\pi
^{\left\vert 2p+m\right\vert }}=0\text{, \ \ \ }m\geq 0.  \label{kzero}
\end{equation}%
For $k\geq 1$, a calculation using integration by parts gives:
\begin{eqnarray*}
\left\langle F,e_{k,m}\right\rangle _{\mathcal{F}(%
\mathbb{C}
^{d})} &=&\int_{%
\mathbb{C}
^{d}}F(z)\overline{e_{k,m}(z)}e^{-\pi \left\vert z\right\vert ^{2}}dz \\
&=&\int_{%
\mathbb{C}
^{d}}e^{-\pi \left\vert z\right\vert ^{2}}\overline{e_{m}(z)}%
p...(p-k+1)\sum_{k\leq p\leq n-1}\overline{z}^{p-k}\sum_{j\geq
0}c_{j,p}z^{j}dz \\
&=&\sum_{k\leq p\leq n-1}\sum_{j\geq 0}c_{j,p}\frac{p...(p-k+1)\pi
^{\left\vert m\right\vert }}{\sqrt{m!}}\int_{%
\mathbb{C}
^{d}}z^{j}\overline{z}^{m+p-k}e^{-\pi \left\vert z\right\vert ^{2}}dz.
\end{eqnarray*}%
As a result,
\begin{equation*}
\sum_{k\leq p\leq n-1}\frac{p...(p-k+1)(p+m-k)!}{\pi ^{\left\vert
m+2p-2k\right\vert }\sqrt{m!}}c_{m+p-k,p}=0\text{, \ \ }m\geq 0\text{, }%
0\leq k\leq n-1\text{,}
\end{equation*}%
resulting in a triangular system for each $m$. Solving this system we obtain
$c_{j,p}=0$ for $k\leq p\leq n-1$ and $j\geq 0$. Therefore, $F=0$.
\end{proof}

\begin{remark}
Clearly the orthogonality in Proposition 2 can be obtained directly by
integration by parts and moving to polar coordinates. For $k=0$ this has the
advantage of showing that the functions are also orthogonal in the polydisk,
providing the useful "double orthogonality" property as in the proof of \cite%
[Theorem 3.4.2.]{Charly}. However, for $k\geq 0$, the boundary behavior
required in the integration by parts eliminates this advantage, making the
functions $e_{k,m}$ less likely to possess such a property.
\end{remark}

\begin{remark}
It is clear that these functions are reminiscent of the so-called special
Hermite functions, which are the Wigner transforms of two Hermite functions
\cite{Thang}. They also appear in the study of Landau levels in \cite%
{LuefGosson}.
\end{remark}

\begin{definition}
The \emph{true} polyanalytic Fock space of order\emph{\ }$n$ is defined as
\begin{equation}
\mathcal{F}^{n}(%
\mathbb{C}
^{d})=Span\left[ \left\{ e_{n,m}(z)\right\} _{m\geq 0.}\right] \text{.}
\label{true}
\end{equation}
\end{definition}

\begin{remark}
Observe that%
\begin{equation*}
\left( \frac{d}{dz}\right) ^{n}\left[ e^{-\pi \left\vert z\right\vert
^{2}}z^{m}\right] =\frac{d^{m+n}}{dz^{n}d\overline{z}^{m}}\left[ e^{-\pi
\left\vert z\right\vert ^{2}}\right] .
\end{equation*}%
Therefore, our functions $e_{n,m}$ are essentially the complex Hermitian
functions introduced in \cite[pag. 126]{Shigekawa} and, as a result,
according to Theorem 7.1 in \cite{Shigekawa}, the true polyanalytic Fock
spaces are the eigenspaces of the Schr\"{o}dinger operator with magnetic
field in $%
\mathbb{R}
^{2}$, associated with the eigenvalue $n+\frac{1}{2}$. Also, observe that
the basis used in \cite{Ramazanov} approaches this one by a formal limit
procedure.
\end{remark}

The orthogonal basis property has the following consequence:

\begin{corollary}
The polyanalytic Fock space, $\mathbf{F}^{n}(%
\mathbb{C}
^{d})$, admits the following decomposition in terms of true polyanalytic
Fock spaces $\mathcal{F}^{k}(%
\mathbb{C}
^{d})$.
\begin{equation}
\mathbf{F}^{n}(%
\mathbb{C}
^{d})=\mathcal{F}^{0}(%
\mathbb{C}
^{d})\oplus ...\oplus \mathcal{F}^{n-1}(%
\mathbb{C}
^{d})\text{.}  \label{decomposition}
\end{equation}
\end{corollary}

This results in a definition equivalent to the one in \cite{VasiBergman},
where the spaces were defined using the decomposition. Observe that $\mathbf{%
F}^{1}(%
\mathbb{C}
^{d})=\mathcal{F}^{0}(%
\mathbb{C}
^{d})=\mathcal{F}(%
\mathbb{C}
^{d})$ and that functions in $\mathcal{F}^{n}(%
\mathbb{C}
^{d})$ are polyanalytic of order $n+1$.

\subsection{Unitarity of $\mathcal{B}^{n}$}

The true polyanalytic Bargmann transform keeps the unitarity property

\begin{theorem}
The true polyanalytic Bargmann transform is an isometric isomorphism
\begin{equation*}
\mathcal{B}^{n}:L^{2}(%
\mathbb{R}
^{d})\rightarrow \mathcal{F}^{n}(%
\mathbb{C}
^{d})\text{.}
\end{equation*}
\end{theorem}

\begin{proof}
Since we know from (\ref{polybfisometry}) that $\mathcal{B}^{n}$ is
isometric, we only need to show that $\mathcal{B}^{n}[L^{2}(%
\mathbb{R}
^{d})]$ is dense in $\mathcal{F}^{n}(%
\mathbb{C}
^{d})$. This is now easy, since the Hermite functions constitute a basis of $%
L^{2}(%
\mathbb{R}
^{d})$ and, by (\ref{Barbasis}), they are mapped onto the basis $\left\{
e_{n,m}(z)\right\} $ of $\mathcal{F}^{n}(%
\mathbb{C}
)$. Since $\mathcal{B}^{n}[L^{2}(%
\mathbb{R}
^{d})]$ contains a basis of $\ \mathcal{F}^{n}(%
\mathbb{C}
^{d})$, then it must be dense.
\end{proof}

\subsection{The polyanalytic Bargmann transform}

Now consider the Hilbert space $\mathcal{H}=L^{2}(%
\mathbb{R}
^{d},%
\mathbb{C}
^{n})$ consisting of vector-valued functions $\mathbf{f}=(f_{0},...,f_{n-1})$
with the inner product

\begin{equation}
\left\langle \mathbf{f,g}\right\rangle _{\mathcal{H}}=\sum_{0\leq k\leq
n-1}\left\langle f_{k},g_{k}\right\rangle _{L^{2}(%
\mathbb{R}
^{d})}\text{.}  \label{innersuper}
\end{equation}

The \emph{polyanalytic Bargmann transform }of a function $\mathbf{f}%
=(f_{0},...,f_{n-1})$ is defined as
\begin{equation}
(\mathbf{B}^{n}\mathbf{f})(z)=\sum_{0\leq k\leq n-1}(\mathcal{B}^{k}f_{k})(z)%
\text{.}  \label{polybargmann}
\end{equation}

The next Theorem, which may have independent interest as a generalization of
Bargmann%
\'{}%
s unitary transform, will be the cornerstone in the proof of our main
results regarding sampling and interpolation.

\begin{theorem}
The polyanalytic Bargmann transform is an isometric isomorphism
\begin{equation*}
\mathbf{B}^{n}:\mathcal{H}\rightarrow \mathbf{F}^{n}(%
\mathbb{C}
^{d})\text{.}
\end{equation*}
\end{theorem}

\begin{proof}
For the isometry, first observe that, from (\ref{ortogonalityrelations}) and
(\ref{rel}),
\begin{equation*}
\left\langle \mathcal{B}^{k}f_{k},\mathcal{B}^{j}f_{j}\right\rangle _{%
\mathcal{F}^{n}(%
\mathbb{C}
^{d})}=\delta _{k,j}.
\end{equation*}%
Then, using the isometric property of $\mathcal{B}^{n}$,
\begin{eqnarray*}
\left\Vert \mathbf{B}^{n}\mathbf{f}\right\Vert _{\mathbf{F}^{n}(%
\mathbb{C}
^{d})}^{2} &=&\sum_{0\leq k\leq n-1}\left\Vert \mathcal{B}%
^{k}f_{k}\right\Vert _{\mathcal{F}^{n}(%
\mathbb{C}
^{d})}^{2} \\
&=&\sum_{0\leq k\leq n-1}\left\Vert f_{k}\right\Vert _{L^{2}(%
\mathbb{R}
^{d})}^{2}=\left\Vert f\right\Vert _{\mathcal{H}}^{2}\text{.}
\end{eqnarray*}%
Moreover, $\mathbf{B}^{n}[L^{2}(%
\mathbb{R}
^{d})]$ is dense in $\mathbf{F}^{n}(%
\mathbb{C}
^{d})$, since, by the decomposition (\ref{decomposition}), every element $%
F\in \mathbf{F}^{n}(%
\mathbb{C}
^{d})$ can be written as
\begin{equation*}
F=\sum_{0\leq k\leq n-1}F_{k}\text{,}
\end{equation*}%
with $F_{k}\in \mathcal{F}^{k}(%
\mathbb{C}
^{d})$, $0\leq k\leq n-1$. Since $\mathcal{B}^{k}$ is unitary, there exists $%
f_{k}\in L^{2}(%
\mathbb{R}
^{d})$ such that $F_{k}=\mathcal{B}^{k}f_{k}$, for every $0\leq k\leq n-1$.
It follows that $F=\mathbf{B}^{n}\mathbf{f}$, with $\mathbf{f=}%
(f_{0},...,f_{n-1})$.
\end{proof}

\section{Sampling in $\mathbf{F}^{n}(%
\mathbb{C}
)$}

\subsection{Definitions}

We will work with lattices. A lattice is a discrete subgroup in $%
\mathbb{R}
^{2d}$ of the form $\Lambda =A%
\mathbb{Z}
^{2d}$, where $A$ is an invertible $2d\times 2d$ matrix. We will define the
\emph{density} of the lattice by%
\begin{equation}
D(\Lambda )=\frac{1}{\left\vert \det A\right\vert }\text{.}
\label{densitylattice}
\end{equation}%
If we write $z=x+i\omega $ and $\pi _{z}g=M_{\omega }T_{x}g$, the adjoint
lattice $\Lambda ^{0}$ is defined by the commuting property as%
\begin{equation*}
\Lambda ^{0}=\{\mu \in
\mathbb{R}
^{2d}:\pi _{z}\pi _{\mu }=\pi _{\mu }\pi _{z}\text{, for all }z\in \Lambda \}%
\text{.}
\end{equation*}%
If $\Lambda =\alpha
\mathbb{Z}
\times \beta
\mathbb{Z}
$, then $\Lambda ^{0}=\beta ^{-1}%
\mathbb{Z}
\times \alpha ^{-1}%
\mathbb{Z}
$. In general,
\begin{equation*}
\Lambda ^{0}=\mathcal{J}(A^{T})^{-1}%
\mathbb{Z}
^{d}\text{,}
\end{equation*}%
where $A^{T}$ is the transpose of $A$ and $\mathcal{J=}\left[
\begin{array}{cc}
0 & I \\
-I & 0%
\end{array}%
\right] $ (consisting of $d\times d$ blocks) is the matrix defining the
standard sympletic form (see \cite{CharlyYurasuper} and \cite{FK}).
Therefore,
\begin{equation}
D(\Lambda ^{0})=\frac{1}{D(\Lambda )}\text{.}  \label{reldensities}
\end{equation}%
We will use the notation $\Gamma =\{z=x+i\omega \}$ to indicate the complex
sequence associated with the sequence $\Lambda =(x,\omega )$. The density of
$\Gamma $ will be the density of the associated lattice, that is $D(\Gamma
)=D(\Lambda )$.

\begin{definition}
$\Gamma $ is a sampling sequence for the space $\mathbf{F}^{n}(%
\mathbb{C}
^{d})$\ if there exist positive constants $A$ and $B$ such that, for every $%
F\in \mathbf{F}^{n}(%
\mathbb{C}
^{d})$,
\begin{equation}
A\left\Vert F\right\Vert _{\mathbf{F}^{n}(%
\mathbb{C}
^{d})}^{2}\leq \sum_{z\in \Gamma }\left\vert F(z)\right\vert ^{2}e^{-\pi
\left\vert z\right\vert ^{2}}\leq B\left\Vert F\right\Vert _{\mathbf{F}^{n}(%
\mathbb{C}
^{d})}^{2}.  \label{sampling}
\end{equation}
\end{definition}

The definition of sampling in the spaces $\mathcal{F}^{k}(%
\mathbb{C}
^{d})$ is exactly the same.

Now we take the following definition, obtained from \cite[page 114]%
{Brekkeseip}, by making a small simplification (in the notation of \cite[%
page 114]{Brekkeseip} we set $\nu (z)=n$ ) and rewriting it in our context
(observe that the weight $e^{i\pi x\omega }$ makes no diference).

\begin{definition}
A sequence $\Gamma _{n}$, consisting of $n$ copies of $\Gamma $ is a
multiple interpolating sequence in the Fock space $\mathcal{F}(%
\mathbb{C}
^{d})$ if, for every sequence $\{\alpha _{i,j}^{(k)}\}_{k=0,...n-1}$ such
that $\{\alpha _{i,j}^{(k)}\}_{k=0,...n-1}\in l^{2}$, there exists $F\in
\mathcal{F}(%
\mathbb{C}
^{d})$ such that $\left\langle F,\beta _{z}e_{k}\right\rangle =\alpha
_{i,j}^{(k)},$ for all $0\leq k\leq n-1$ and every $z\in \Gamma $.
\end{definition}

Consider again the Hilbert space $\mathcal{H}=L^{2}(%
\mathbb{R}
^{d},%
\mathbb{C}
^{n})$ consisting of vector-valued functions $\mathbf{f}=(f_{0},...,f_{n-1})$
with the inner product (\ref{innersuper}). The time-frequency shifts act
coordinate-wise in an obvious way.

\begin{definition}
The vector valued system $\mathcal{G}(\mathbf{g},\Lambda )=\{M_{\omega }T_{x}%
\mathbf{g}\}_{(x,w)\in \Lambda }$ is a $\emph{Ga}$\emph{b}$\emph{or}$ \emph{%
superframe} for $\mathcal{H}$ if there exist constants $A$ and $B$ such
that, for every $\mathbf{f}\in \mathcal{H}$,
\begin{equation}
A\left\Vert \mathbf{f}\right\Vert _{\mathcal{H}}^{2}\leq \sum_{(x,w)\in
\Lambda }\left\vert \left\langle \mathbf{f},M_{\omega }T_{x}\mathbf{g}%
\right\rangle _{\mathcal{H}}\right\vert ^{2}\leq B\left\Vert \mathbf{f}%
\right\Vert _{\mathcal{H}}^{2}.  \label{superframe}
\end{equation}
\end{definition}

Superframes were introduced in a more abstract form in \cite{DL} and in the
context of "multiplexing" in \cite{Balan}. We will need a fundamental
structure Theorem from time-frequency analysis, namely the following version
of the Janssen-Ron-Shen duality principle \cite[Theorem 2.7]{CharlyYurasuper}%
.%
\begin{equation*}
\end{equation*}

\textbf{Theorem A. }Let $\mathbf{g}=(g_{0},...,g_{n-1})$.\ The vector valued
system $\mathcal{G}(\mathbf{g},\Lambda )$ is a $\emph{Ga}$\emph{b}$\emph{or}$
\emph{superframe} for $\mathcal{H}$ if and only if the union of Gabor
systems $\cup _{k=0}^{n-1}\mathcal{G}(g_{k},\Lambda ^{0})$ is a Riesz
sequence for $L^{2}(%
\mathbb{R}
^{d})$.

\subsection{Duality principle}

In this section we will obtain the following duality principle.

\begin{theorem}
$\Gamma $ is a sampling sequence for $\mathbf{F}^{n}(%
\mathbb{C}
^{d})$ if and only if the adjoint sequence $\Gamma _{n}^{0}$ is a multiple
interpolating sequence in the Fock space $\mathcal{F}(%
\mathbb{C}
^{d})$.
\end{theorem}

We first prove two Lemmas. Combining them with Theorem A gives Theorem 3.

\begin{lemma}
The union of Gabor systems $\cup _{k=0}^{n-1}\mathcal{G}(h_{k},\Lambda )$ is
a Riesz sequence for $L^{2}(%
\mathbb{R}
^{d})$ if and only if $\Gamma _{n}$ is a multiple interpolating sequence in
the Fock space $\mathcal{F}(%
\mathbb{C}
^{d})$.
\end{lemma}

\begin{proof}
The union of Gabor systems $\cup _{k=0}^{n-1}\mathcal{G}(h_{k},\Lambda )$ is
a Riesz sequence for $L^{2}(%
\mathbb{R}
^{d})$ if, for every sequence $\{\alpha _{i,j}^{(k)}\}_{k=0,...n-1}\in l^{2}$%
, there exists a $f\in L^{2}(%
\mathbb{R}
^{d})$ such that $\left\langle f,M_{\omega }T_{x}h_{k}\right\rangle =\alpha
_{i,j}^{(k)}$, for all $k=0,...n-1$ and every $(x,\omega )\in \Lambda $.
Using the unitarity of $\mathcal{B}$ and the intertwining property (\ref%
{intertwining}) gives
\begin{equation*}
\left\langle f,M_{\omega }T_{x}h_{k}\right\rangle =\left\langle \mathcal{B}%
f,\beta _{z}e_{k}\right\rangle \text{,}
\end{equation*}%
and setting $F=\mathcal{B}f$ shows that $\Gamma _{n}$ is a multiple
interpolating sequence in the Fock space $\mathcal{F}(%
\mathbb{C}
^{d})$.
\end{proof}

The next Lemma is a key step in our argument and it is at this point that
the unitarity of the polyanalytic Bargmann transform is essential.

\begin{lemma}
Let $\mathbf{h}_{n}=(h_{0},...,h_{n-1})$. Then the set $\mathcal{G}(\mathbf{h%
}_{n},\Lambda )$ is a Gabor superframe for $\mathcal{H}=L^{2}(%
\mathbb{R}
^{d},%
\mathbb{C}
^{n})$ if and only if the associated complex sequence $\Gamma $ is a
sampling sequence for $\mathbf{F}^{n}(%
\mathbb{C}
^{d})$.
\end{lemma}

\begin{proof}
Using the definition of the inner product (\ref{innersuper}), identity (\ref%
{rel}) and the definition of the polyanalytic Bargmann transform, it is
clear that
\begin{eqnarray}
\left\langle \mathbf{f},M_{\omega }T_{x}\mathbf{h}_{n}\right\rangle _{%
\mathcal{H}} &=&\sum_{0\leq k\leq n-1}\left\langle f_{k},M_{\omega
}T_{x}h_{k}\right\rangle _{L^{2}(%
\mathbb{R}
^{d})}  \label{1} \\
&=&\sum_{0\leq k\leq n-1}e^{i\pi x\omega -\frac{\pi }{2}\left\vert
z\right\vert ^{2}}(\mathcal{B}^{k}f_{k})(z)  \notag \\
&=&e^{i\pi x\omega -\frac{\pi }{2}\left\vert z\right\vert ^{2}}(\mathbf{B}%
^{n}\mathbf{f})(z)\text{.}  \label{3}
\end{eqnarray}%
Therefore, setting $F=\mathcal{B}^{n}\mathbf{f}$, the unitarity of $\mathbf{B%
}^{n}$ shows that\ the inequalities (\ref{superframe}) are equivalent to (%
\ref{sampling}).
\end{proof}

\subsection{Main result}

We will need the concept of \emph{Beurling density} of a sequence.

Let $I$ be a compact set of measure $1$ in the complex plane and let $%
n^{-}(r)$ denote the smallest (and $n^{+}(r)$ the biggest) number of points
from $\Gamma $ to be found in a translate of $rI$. We define the \emph{lower}
and the \emph{upper} Beurling density of $\Gamma $ to be%
\begin{equation*}
D^{-}(\Gamma )=\lim_{r\rightarrow \infty }\sup \frac{n^{-}(r)}{r^{2}}\text{
and }D^{+}(\Gamma )=\lim_{r\rightarrow \infty }\sup \frac{n^{+}(r)}{r^{2}}%
\text{,}
\end{equation*}%
respectively. When $\Gamma $ is associated with the lattice $\Lambda $, $%
D^{-}(\Gamma )=D^{+}(\Gamma )=D(\Gamma )=D(\Lambda )$.

We will now use the following result, which is Theorem 2.2 in \cite%
{Brekkeseip}. Observe that we can remove the uniformly discrete condition
from the statement in \cite{Brekkeseip} since we are dealing only with
lattices.%
\begin{equation*}
\end{equation*}%
\textbf{Theorem B. }The sequence\textbf{\ }$\Gamma _{n}$ is a multiple
interpolating lattice sequence in the Fock space $\mathcal{F}(%
\mathbb{C}
)$ if and only if $D(\Gamma _{n})<1$.%
\begin{equation*}
\end{equation*}

From this we obtain the characterization of sampling lattices in $\mathbf{F}%
^{n}(%
\mathbb{C}
)$.

\begin{theorem}
The lattice $\Gamma $ is a sampling sequence for $\mathbf{F}^{n}(%
\mathbb{C}
)$ if and only if $D(\Gamma )>n$.
\end{theorem}

\begin{proof}
We know by the duality principle that $\Gamma $ is a sampling sequence for $%
\mathbf{F}^{n}(%
\mathbb{C}
)$ if and only if the adjoint sequence $\Gamma _{n}^{0}$ is a multiple
interpolating sequence in the Fock space $\mathcal{F}(%
\mathbb{C}
)$. By definition of Beurling density, it is obvious that%
\begin{equation*}
D(\Gamma _{n}^{0})=nD(\Gamma ^{0})\text{.}
\end{equation*}%
Therefore, Theorem B states that $\Gamma ^{0}$ is a multiple interpolating
sequence in the Fock space $\mathcal{F}(%
\mathbb{C}
)$ if and only if $D(\Gamma ^{0})<\frac{1}{n}.$ Using (\ref{reldensities}),
we conclude that $\Gamma $ is a sampling sequence for $\mathbf{F}^{n}(%
\mathbb{C}
)$ if and only if $D(\Gamma )>n$.
\end{proof}

Using Lemma 1, we recover Theorem 1.1 of \cite{CharlyYurasuper}.

\begin{corollary}
Let $\mathbf{h}_{n}=(h_{0},...,h_{n-1})$. Then the set $\mathcal{G}(\mathbf{h%
}_{n},\Lambda )$ is a Gabor super frame for $\mathcal{H}=L^{2}(%
\mathbb{R}
,%
\mathbb{C}
^{n})$ if and only if $D(\Gamma )>n$.
\end{corollary}

\section{Interpolation in $\mathbf{F}^{n}(%
\mathbb{C}
)$}

\subsection{Definitions}

\begin{definition}
The sequence $\Gamma $ is an interpolating sequence for $\mathbf{F}^{n}(%
\mathbb{C}
^{d})$ if, for every sequence $\{\alpha _{i,j}\}\in l^{2}$, there exists $%
F\in \mathbf{F}^{n}(%
\mathbb{C}
^{d})$ such that $e^{i\pi x\omega -\frac{\pi }{2}\left\vert z\right\vert
^{2}}F(z)=\alpha _{i,j},$ for every $z\in \Gamma $.
\end{definition}

\begin{definition}
The sequence $\Gamma _{n}$, consisting of $n$ copies of $\Gamma $ is is said
to be a multiple sampling sequence for $\mathcal{F}(%
\mathbb{C}
^{d})$ if there exist numbers $A$ and $B$ such that%
\begin{equation}
A\left\Vert F\right\Vert _{\mathcal{F}(%
\mathbb{C}
^{d})}^{2}\leq \sum_{z\in \Gamma }\sum_{0\leq k\leq n-1}\left\vert
\left\langle F,\beta _{z}e_{k}\right\rangle \right\vert ^{2}\leq B\left\Vert
F\right\Vert _{\mathcal{F}(%
\mathbb{C}
^{d})}^{2}.  \label{multisampling}
\end{equation}
\end{definition}

\begin{definition}
The set $\cup _{k=0}^{n-1}\mathcal{G}(g_{k},\Lambda )$ is said to generate a
\emph{Gabor multi-frame} in $L^{2}(%
\mathbb{R}
^{d})$\ if there exist constants $A$ and $B$ such that, for every $f\in
L^{2}(%
\mathbb{R}
^{d})$,%
\begin{equation}
A\left\Vert f\right\Vert _{L^{2}(%
\mathbb{R}
^{d})}^{2}\leq \sum_{(x,\omega )\in \Lambda }\sum_{0\leq k\leq
n-1}\left\vert \left\langle f,M_{\omega }T_{x}g_{k}\right\rangle _{L^{2}(%
\mathbb{R}
^{d})}\right\vert ^{2}\leq B\left\Vert f\right\Vert _{L^{2}(%
\mathbb{R}
^{d})}^{2}.  \label{multiframe}
\end{equation}
\end{definition}

The dual of the duality principle contained in Theorem A is now required. As
stated at the end of \cite{CharlyIMRN}, it reads as follows:%
\begin{equation*}
\end{equation*}%
\textbf{Theorem C. }The set $\mathcal{G}(\mathbf{g},\Lambda )$ is a Riesz
sequence for $L^{2}(%
\mathbb{R}
^{d})$ if and only if $\cup _{k=0}^{n-1}\mathcal{G}(g_{k},\Lambda ^{0})$ is
a Gabor multi-frame in $L^{2}(%
\mathbb{R}
^{d})$.

\subsection{Duality principle}

Now we prove the following duality.

\begin{theorem}
The sequence $\Gamma $ is an interpolating sequence for $\mathbf{F}^{n}(%
\mathbb{C}
^{d})$ if and only if $\Gamma _{n}^{0}$ is a multiple sampling sequence for $%
\mathcal{F}(%
\mathbb{C}
^{d})$.
\end{theorem}

As in the sampling section, we first prove two Lemmas which, combined with
Theorem C, give the result. The next Lemma requires only the unitarity of
the Bargmann transform.

\begin{lemma}
The set $\cup _{k=0}^{n-1}\mathcal{G}(h_{k},\Lambda )$ is a Gabor
multi-frame in $L^{2}(%
\mathbb{R}
^{d})$\ if and only if $\Gamma _{n}$ is a multiple sampling sequence for $%
\mathcal{F}(%
\mathbb{C}
^{d})$.

\begin{proof}
Similar to Lemma 1: using the unitarity of $\mathcal{B}$ and the
intertwining property (\ref{intertwining}) gives $\left\langle f,M_{\omega
}T_{x}h_{k}\right\rangle =\left\langle \mathcal{B}f,\beta
_{z}e_{k}\right\rangle $; setting $F=\mathcal{B}f$ it follows from the
unitarity of the Bargmann transform that (\ref{multisampling}) and (\ref%
{multiframe}) are equivalent.
\end{proof}
\end{lemma}

Again, we make the key connection in the next step, where the unitarity of
the polyanalytic Bargmann transform is required.

\begin{lemma}
The sequence $\Gamma $ is an interpolating sequence for $\mathbf{F}^{n}(%
\mathbb{C}
^{d})$ if and only if $\mathcal{G}(\mathbf{h}_{n},\Lambda )$ is a Riesz
sequence for $\mathcal{H}$.
\end{lemma}

\begin{proof}
The sequence $\Gamma $ is an interpolating sequence for $\mathbf{F}^{n}(%
\mathbb{C}
^{d})$ if, for every sequence $\{\alpha _{i,j}\}\in l^{2}$, there exists $%
F\in \mathbf{F}^{n}(%
\mathbb{C}
^{d})$ such that $e^{i\pi x\omega -\frac{\pi }{2}\left\vert z\right\vert
^{2}}F(z)=\alpha _{i,j},$ for every $z\in \Gamma $. Using the unitarity of $%
\mathbf{B}^{n}$, we find, for every $F\in \mathbf{F}^{n}(%
\mathbb{C}
^{d})$, a vector valued function $\mathbf{f}\in \mathcal{H}\ $such that $%
\mathbf{B}^{n}\mathbf{f}=F$ or, by (\ref{1})-(\ref{3}),$\ \left\langle
\mathbf{f},M_{\omega }T_{x}\mathbf{h}_{n}\right\rangle _{\mathcal{H}%
}=e^{i\pi x\omega -\frac{\pi }{2}\left\vert z\right\vert ^{2}}F$. Therefore,
the first assertion is equivalent to say that, for every sequence $\{\alpha
_{i,j}\}\in l^{2}$, there exists a $\mathbf{f}\in \mathcal{H}$ such that $%
\left\langle \mathbf{f},M_{\omega }T_{x}\mathbf{h}_{n}\right\rangle _{%
\mathcal{H}}=\alpha _{i,j}$, for every $z\in \Gamma $. This says that $%
\mathcal{G}(\mathbf{h}_{n},\Lambda )$ is a Riesz sequence for $\mathcal{H}$.
\end{proof}

\subsection{Main result}

We will need the following result, which is contained in Theorem 2.1 in \cite%
{Brekkeseip}:%
\begin{equation*}
\end{equation*}

\textbf{Theorem D. }The sequence\textbf{\ }$\Gamma _{n}$ is a multiple
interpolating sequence in the Fock space $\mathcal{F}(%
\mathbb{C}
)$ if and only if $D(\Gamma _{n})<1$.%
\begin{equation*}
\end{equation*}

As before, we can obtain our main result from this one.

\begin{theorem}
The lattice $\Gamma $ is an interpolating sequence for $\mathbf{F}^{n}(%
\mathbb{C}
)$ if and only if $D(\Gamma )<n$.
\end{theorem}

\begin{proof}
We know from the duality principle that $\Gamma $ is an interpolating
sequence for $\mathbf{F}^{n}(%
\mathbb{C}
)$ if and only if $\Gamma _{n}$ is a multiple sampling sequence for $%
\mathcal{F}(%
\mathbb{C}
)$. Once again we have $D(\Gamma _{n}^{0})=nD(\Gamma ^{0})$. Therefore,
Theorem D states that $\Gamma ^{0}$ is a multiple interpolating sequence in
the Fock space $\mathcal{F}(%
\mathbb{C}
)$ if and only if $D(\Gamma ^{0})>\frac{1}{n}$. As in Theorem 5 it follows
that $\Gamma $ is an interpolating sequence for $\mathbf{F}^{n}(%
\mathbb{C}
)$ if and only if $D(\Gamma )<n$.
\end{proof}

From this and Lemma 4 we obtain a new result characterizing all the lattices
which generate vector valued Gabor Riesz sequences with Hermite functions$.$
This reveals, at least for lattices, the existence of a critical density for
vector-valued Gabor systems with Hermite functions.

\begin{theorem}
$\mathcal{G}(\mathbf{h}_{n},\Lambda )$ is a Riesz sequence for $\mathcal{H}$
if and only if $D(\Gamma )<n$.
\end{theorem}

\begin{remark}
We should remark that the reason we did not care about the Bessel condition
in the equivalence of the Riesz sequence and interpolating property, used
several times in the previous section, is that the Hermite functions belong
to Feichtinger%
\'{}%
s algebra $S_{0}$ (see \cite{FeiZim},\cite{Fei}):%
\begin{equation*}
\left\Vert h_{n}\right\Vert _{S_{0}}=\int_{%
\mathbb{R}
}\left\vert \left\langle h_{n}\text{,}M_{\omega }T_{x}\varphi \right\rangle
\right\vert dz<\infty \text{,}
\end{equation*}%
where $\varphi $ is the $L^{2}$-normalized Gaussian. As a result they
satisfy the Bessel condition \cite[theorem 12]{Heil}.
\end{remark}

\section{Generalizations, applications and open problems}

\subsection{The super Gabor transform}

The polyanalytic Bargmann transform is an instance of a more general
transform. Although it plays no role in the proofs of our main results, we
briefly describe it here for completeness of the picture.

The \emph{super Gabor transform} of a function $\mathbf{f}$ with respect to
the \textquotedblright window\textquotedblright\ $\mathbf{g=(}%
g_{0},...g_{n-1}\mathbf{)}$ is defined, for every $x,\omega \in
\mathbb{R}
^{d}$, as
\begin{equation}
\mathbf{V}_{\mathbf{g}}\mathbf{f}(x,\omega )=\left\langle \mathbf{f,}M%
\mathbf{_{\omega }}T_{x}\mathbf{g}\right\rangle _{\mathcal{H}}=\sum_{0\leq
k\leq n-1}\left\langle f_{k},M_{\omega }T_{x}g_{k}\right\rangle _{L^{2}(%
\mathbb{R}
^{d})}.  \label{supergabor}
\end{equation}

That is to say,%
\begin{equation*}
\mathbf{V}_{\mathbf{g}}\mathbf{f}(x,\omega )=\sum_{0\leq k\leq
n-1}V_{g_{k}}f_{k}(x,\omega )\text{.}
\end{equation*}%
This defines a map%
\begin{equation*}
\mathbf{V}_{\mathbf{g}}\mathbf{f:}\mathcal{H\rightarrow }L^{2}(%
\mathbb{R}
^{2d})\text{.}
\end{equation*}

In the case when the vector $\mathbf{g}$ is extracted from an orthogonal
sequence $\{g_{k}\}_{k\geq 0}$, the essential properties of the Gabor
transform are kept. As an example of how the results concerning Gabor
analysis can be generalized to this setting, we obtain the isometric
properties and the orthogonality relations (the latter valid under the
slightly weaker condition of biorthogonality).

\begin{proposition}
If
\begin{equation}
\left\langle g_{i},g_{j}\right\rangle _{L^{2}(%
\mathbb{R}
^{d})}=\delta _{i,j},  \label{ortg}
\end{equation}%
then $\mathbf{V}_{\mathbf{g}}\mathbf{f}$ is an isometry between Hilbert
spaces, that is%
\begin{equation*}
\left\Vert \mathbf{V}_{g}\mathbf{f}\right\Vert _{L^{2}(%
\mathbb{R}
^{2d})}=\left\Vert \mathbf{f}\right\Vert _{\mathcal{H}}.
\end{equation*}
\end{proposition}

\begin{proof}
Using (\ref{ortg}) and (\ref{Gabor isometry}) gives%
\begin{eqnarray*}
\left\Vert \mathbf{V}_{g}\mathbf{f}\right\Vert _{L^{2}(%
\mathbb{R}
^{2d})}^{2} &=&\sum_{0\leq k\leq n-1}\left\langle V_{g_{k}}f_{k}\mathbf{,}%
V_{g_{k}}f_{k}\right\rangle _{L^{2}(%
\mathbb{R}
^{2d})} \\
&=&\sum_{0\leq k\leq n-1}\left\Vert V_{g_{k}}f_{k}\right\Vert _{L^{2}(%
\mathbb{R}
^{2d})}^{2} \\
&=&\sum_{0\leq k\leq n-1}\left\Vert f_{k}\right\Vert _{L^{2}(%
\mathbb{R}
^{d})}^{2} \\
&=&\left\Vert \mathbf{f}\right\Vert _{\mathcal{H}}^{2}.
\end{eqnarray*}
\end{proof}

The orthogonality relations are valid under the slightly weaker condition of
biorthogonality.

\begin{proposition}
If%
\begin{equation}
\left\langle g_{1,i},g_{2,j}\right\rangle _{L^{2}(%
\mathbb{R}
^{d})}=\delta _{i,j},  \label{biort}
\end{equation}%
then $\mathbf{V}_{g}\mathbf{f}$ satisfies%
\begin{equation}
\left\langle \mathbf{V}_{g_{1,i}}\mathbf{f}_{1},\mathbf{V}_{g_{2,j}}\mathbf{f%
}_{2}\right\rangle _{L^{2}(%
\mathbb{R}
^{2d})}=\left\langle \mathbf{f}_{1}\mathbf{,f}_{2}\right\rangle _{\mathcal{H}%
}\text{.}  \label{vectorort}
\end{equation}
\end{proposition}

\begin{proof}
Using (\ref{biort}) and (\ref{Gabor isometry}) gives%
\begin{equation*}
\left\langle \mathbf{V}_{g_{1,i}}\mathbf{f}_{1},\mathbf{V}_{g_{2,j}}\mathbf{f%
}_{2}\right\rangle _{L^{2}(%
\mathbb{R}
^{2d})}=\sum_{0\leq k\leq n-1}\left\langle \mathbf{f}_{1,k},\mathbf{f}%
_{2,k}\right\rangle _{L^{2}(%
\mathbb{R}
^{2d})}=\left\langle \mathbf{f}_{1}\mathbf{,f}_{2}\right\rangle _{\mathcal{H}%
}\text{.}
\end{equation*}
\end{proof}

Clearly, when we take $\mathbf{g=(}\Phi _{0},...\Phi _{m-1}\mathbf{)}$, we
have the following relation with the polyanalytic Bargmann transform:%
\begin{equation*}
(\mathbf{B}^{n}\mathbf{f})(z)=e^{i\pi x\omega -\pi \frac{\left\vert
z\right\vert ^{2}}{2}}\mathbf{V}_{\mathbf{g}}\mathbf{f}(x,\omega )\text{.}
\end{equation*}%
This observation removes some of the mystery from the previous sections. Now
we are in a position to say that the polyanalytic Bargmann-Fock spaces play
exactly the same role as the Bargmann-Fock space in the scalar case. We have
thus all the basic ingredients to build a theory of vector valued (or super)
Gabor analysis:

\begin{itemize}
\item A vector valued Gabor transform.

\item A discrete theory for $L^{2}(%
\mathbb{R}
^{d})$ frames and Riesz basis.

\item A special vector of windows providing the connection with complex
analysis, where, in the case $d=1$, a Nyquist rate phenomenon can be
observed.
\end{itemize}

The analyzing vector can be extracted from orthogonal systems other than the
Hermite functions, though they probably do not lead to very structured
situations. As a first example we may think of the Haar basis. Other
wavelets may also be used, but we will not pursue this question further in
this paper.

\subsection{Applications}

Although we are here dealing mostly with questions of a conceptual nature,
there are potential applications of these results in multiplexing, an
important method in telecommunications, computer networks and digital video,
as indicated in \cite{Balan} and \cite{CharlyYurasuper}. The idea of
multiplexing is to encode $n$ independent signals $f_{k}\in L^{2}(%
\mathbb{R}
^{d})$ as a single sequence $\mathbf{f}$ that captures the time-frequency
information of each one. With suitable windows\ $\mathbf{g=(}g_{0},...g_{n-1}%
\mathbf{)}$, the time-frequency content of each signal can be measured by
the associated super-Gabor systems. The super Gabor transform $\mathbf{V}_{%
\mathbf{g}}\mathbf{f}(x,\omega )$, gives the precise value one wants to
approximate via the discrete systems.

\subsection{Further questions}

As in the scalar-valued case, when the connection to complex analysis is
missing the lattices generating super-frames should be hard to describe as,
for instance, the case of a "rectangular" window. What we mean by
rectangular window is one obtained from the Haar basis, which is the natural
generalization of the characteristic function of an interval. We wonder how
Janssen%
\'{}%
s tie \cite{J2} would generalize to this situation.

It is still unclear if there is a relation between the Bargmann-Fock space
of polyanalytic functions and vector valued coherent states, as considered
in \cite{Ali} and \cite{AliEngGaz}.

A rather mysterious topic is sampling and interpolation in the true\emph{\ }%
polyanalytic space $\mathcal{F}^{n-1}(%
\mathbb{C}
)$. Partial results follow from ours. Using decomposition (\ref%
{decomposition}), it is easy to see that we obtain two propositions
concerning $\mathcal{F}^{n-1}(%
\mathbb{C}
)$: one, a sufficient condition for sampling, the other a necessary
condition for interpolation.

\begin{proposition}
If $D(\Lambda )>n$ then $\Gamma $ is a sampling sequence for $\mathcal{F}%
^{n-1}(%
\mathbb{C}
)$.
\end{proposition}

\begin{proof}
This is obvious from decomposition (\ref{decomposition}) because, if $%
D(\Lambda )>n$, then inequality (\ref{sampling}) holds for every $F\in
\mathcal{F}^{n}(%
\mathbb{C}
)$. In particular it also holds for every $F\in \mathbf{F}^{n}(%
\mathbb{C}
)$.
\end{proof}

\begin{proposition}
If $\Gamma $ is an interpolating sequence for $\mathcal{F}^{n-1}(%
\mathbb{C}
)$, then
\begin{equation}
D(\Lambda )<n.  \label{estd}
\end{equation}
\end{proposition}

\begin{proof}
If $D(\Lambda )>n$, then $\Gamma $ is not an interpolating sequence for $%
\mathbf{F}^{n}(%
\mathbb{C}
)$. As a result, there exists $\{\alpha _{i,j}\}\in l^{2}$, such that it is
impossible to find $F\in \mathbf{F}^{n}(%
\mathbb{C}
)$, verifying $e^{i\pi x\omega -\frac{\pi }{2}\left\vert z\right\vert
^{2}}F(z)=\alpha _{i,j}$. Again, from the decomposition (\ref{decomposition}%
), one sees that in particular it is impossible to find $F\in \mathcal{F}%
^{n}(%
\mathbb{C}
)$ verifying $e^{i\pi x\omega -\frac{\pi }{2}\left\vert z\right\vert
^{2}}F(z)=\alpha _{i,j}$.
\end{proof}

We may wonder whether these conditions are sharp.

In the case of Proposition 6 the answer is a definite "no". This is because,
if $\Gamma $ is an interpolating sequence for $\mathcal{F}^{n-1}(%
\mathbb{C}
)$, then $\{\left\langle f,M_{\omega }T_{x}h_{n-1}\right\rangle _{L^{2}(%
\mathbb{R}
)}\}_{(x,w)\in \Lambda }$ is a Riesz basis for $L^{2}(%
\mathbb{R}
)$. As a result, $\{\left\langle f,M_{\omega }T_{x}h_{n-1}\right\rangle
_{L^{2}(%
\mathbb{R}
)}\}_{(x,w)\in \Lambda ^{0}}$ is a frame for $L^{2}(%
\mathbb{R}
)$ and consequently, using Rieffel-Ramanathan-Steger Theorem \cite{Rief},
\cite{RS}, we must have $D(\Lambda ^{0})>1$ and, by scalar Ron-Shen duality,
this implies $D(\Lambda )<1$. Therefore, estimate (\ref{estd}) is far from
being sharp.

Let us look closer at Proposition 5. Here things become rather intriguing.

It is easy to see that the set $\mathcal{G}(h_{n},\Lambda )$ is a Gabor
frame if and only if the associated sequence $\Gamma $ is a sampling
sequence for $\mathcal{F}^{n}(%
\mathbb{C}
)$: since $\left\langle f,M_{\omega }T_{x}h_{n-1}\right\rangle _{L^{2}(%
\mathbb{R}
)}=V_{h_{n}}f(x,w)$, then (\ref{rel}) can be written as
\begin{equation}
\left\langle f,M_{\omega }T_{x}h_{n-1}\right\rangle _{L^{2}(%
\mathbb{R}
)}=e^{i\pi x\omega -\frac{\pi }{2}\left\vert z\right\vert ^{2}}\mathcal{B}%
^{n}f(z).  \label{relsamp}
\end{equation}

By setting $F=\mathcal{B}^{n}f$, the isometry of $\mathcal{B}^{n}:L^{2}(%
\mathbb{R}
)\rightarrow \mathcal{F}^{n}(%
\mathbb{C}
)$ and the relation (\ref{relsamp}) show that definition (\ref{sampling}) is
equivalent to the fact that $\{\left\langle f,M_{\omega
}T_{x}h_{n-1}\right\rangle _{L^{2}(%
\mathbb{R}
)}\}_{(x,w)\in \Lambda }$ is a frame.

Therefore, Proposition 5 is equivalent to the sufficient condition obtained
in \cite{CharlyYura}, where the authors prove that, if $D(\Lambda )>n$, then
$\{\left\langle f,M_{\omega }T_{x}h_{n}\right\rangle _{L^{2}(%
\mathbb{R}
)}\}_{(x,w)\in \Lambda }$ is a frame and give some evidence to support their
conjecture that the result is sharp. If true, this would be quite a
surprising statement, in face of decomposition (\ref{decomposition}).
Moreover, by duality, it would bring estimate (\ref{estd}) down to $%
D(\Lambda )<1/n$.

\textbf{Acknowledgement. }I would like to thank Hans Feichtinger for his
postdoctoral guidance and the NuHAG group at the University of Vienna, where
I first had contact with some of the ideas developed in this paper.
Moreover, specific acknowledgement is due to Karlheinz Gr\"{o}chenig, who
gave generous local seminars with fresh ideas on Gabor frames with Hermite
windows and on the vectorial duality of Gabor frames, and to Yurii
Lyubarskii, who read the first version of the manuscript, providing comments
and corrections which are incorporated in this version. The Referees and the
Editor helped to improve the readability of the paper and they have drawn my
attention to the recent preprint \cite{BCL}. Finally, Radu Balan explained
me the results in \cite{BCL} and gave insightful remarks on the whole topic
while I was preparing the revised version of the manuscript.

\end{document}